\documentclass[reqno]{amsart}
\usepackage{amsfonts}
\usepackage{graphicx}
\usepackage{amscd}
\usepackage{amsmath}
\usepackage{amssymb}
\usepackage{latexsym}
\usepackage{hyperref}
\usepackage[all]{xy}
\usepackage{color}
\usepackage{mathrsfs}
\usepackage[latin1]{inputenc}
\newtheorem{example}{\bf Example}
\newtheorem{lemma}{\bf Lemma}

\newtheorem{theorem}{\bf Theorem}
\numberwithin{equation}{section}

\newcommand\dv{\mathrm{div}}
\begin{document}

\title[A note on gradient Einstein-type manifolds]{A note on gradient Einstein-type manifolds}

\author[José N.V. Gomes]{José Nazareno Vieira Gomes}
\address{Departamento de Matemática, Instituto de Ciências Exatas, Universidade Federal do Amazonas, Av. General Rodrigo Octávio, 6200, 69080-900 Manaus, Amazonas, Brazil.}
\email{jnvgomes@pq.cnpq.br,jnvgomes@gmail.com}
\urladdr{https://ufam.edu.br}
\keywords{Einstein-type manifolds, Rigidity results, Standard sphere.}
\subjclass[2010]{Primary 53C25; Secondary 53C24.}
\thanks{Partially supported by Grants 304608/2015-7 and 202234/2017-7, Conselho Nacional de Desenvolvimento Científico e Tecnológico (CNPq).}

\begin{abstract}
In this note, we show that a nontrivial, compact, degenerate or nondegenerate, gradient Einstein-type manifold of constant scalar curvature is isometric to the standard sphere with a well defined potential function. Moreover, under some geometric assumptions the noncompact case is also treated. In this case, the main result is that a homogeneous, proper, noncompact, nondegenerate, gradient Einstein-type manifold is an Einstein manifold.
\end{abstract}
\maketitle

\section{Introduction}
Let $(M^n,g)$ be a connected Riemannian manifold of dimension $n\geq3$. We denote by $Ric$ and $S$ its corresponding Ricci tensor and scalar curvature, respectively. We say that $(M^n,g)$ is an Einstein-type manifold or that $(M^n,g)$ supports an Einstein-type structure if there exist a smooth vector field $X$ on $M^n$ and a smooth function $\lambda:M^n\to\Bbb{R}$ such that
\begin{equation}\label{1-1}
\alpha Ric + \frac{\beta}{2}\mathscr{L}_Xg+\mu X^\flat\otimes X^\flat = (\rho S +\lambda)g
\end{equation}
for some constants $\alpha, \beta, \mu, \rho \in \Bbb{R}$, with $(\alpha,\beta,\mu)\neq (0,0,0)$. Here $\mathscr{L}$ and $X^\flat$ stand for the Lie derivative and the $1$--form metrically dual to the vector field $X$, respectively. If $X=\nabla f$ for some smooth function $f:M^n\to\Bbb{R}$, we say that $(M^n,g)$ is a gradient Einstein-type manifold. In this case, Eq.~\eqref{1-1} can be rewritten as
\begin{equation}\label{1-2}
\alpha Ric + \beta \nabla^2f+ \mu df\otimes df = (\rho S +\lambda)g,
\end{equation}
where $\nabla^2f$ stands for the Hessian of $f$. We refer to $f$ as the potential function.

The concept of Einstein-type manifold was introduced recently by Catino et al.~\cite{CMMR}. Notice that, in terms of Eq~\eqref{1-1}, an Einstein-type structure on a Riemannian manifold $(M^n,g)$ unifies various particular cases well studied in the literature, such as gradient Ricci solitons, gradient Ricci almost solitons, Yamabe solitons, Yamabe quasisolitons, conformal gradient solitons, $m$--quasi-Einstein manifolds and $\rho$--Einstein solitons. Each of them has a particular importance. Let us begin by making brief comments of three of these latter. Ricci almost solitons were introduced in Pigola et al.~\cite{prrs} merely as generalizations of the equation of Ricci solitons. Recently, it was discovered that a special family of Ricci almost solitons (namely, the $\rho$-Einstein solitons) arises from the Ricci-Bourguignon flow, see Catino et al.~\cite{Catino2} or Catino-Mazzieri~\cite{Catino-Mazzieri}, whereas the traditional Ricci solitons correspond to self-similar solutions to the Ricci flow and often arise as limits of dilations of singularities in the Ricci flow, see Hamilton~\cite{hamilton2}. The notion of an $m$--quasi-Einstein manifold, originated from the study of Einstein warped product manifolds, see Besse~\cite{besse}. More recently, Freitas Filho studied in his PhD thesis~\cite{Airton} a modified Ricci soliton as a class of Einstein type manifolds (or metrics) that contain both Ricci solitons and $m$--quasi-Einstein manifolds. This class is closely related to the construction of Ricci solitons that are realized as warped products. Moreover, a modified Ricci soliton appears as part of a self-similar solution of the modified Harmonic-Ricci flow which results in a new characterization of $m$--quasi-Einstein manifolds.

Coming back to the initial definition, we say that a gradient Einstein-type manifold is nondegenerate if $\beta^2\neq(n-2)\alpha\mu$ and $\beta\neq0$. Otherwise, if $\beta^2=(n-2)\alpha\mu$ and $\beta\neq0$ we have a degenerate gradient Einstein-type manifold. In~\cite{CMMR} the authors give a good justification for this terminology through an equivalence of a degenerate gradient Einstein-type manifold with a conformally Einstein manifold. We recall that a manifold $(M^n,g)$ is conformally Einstein if its metric $g$ can be pointwise conformally deformed to an Einstein metric $\tilde g$. The $\beta=0$ case was addressed separately in \cite{CMMR}. Remarkably, noncompact, nondegenerate, gradient Einstein-type manifolds can be locally characterized  when its Bach tensor is null. In particular, in the cases of $\rho$-Einstein solitons and Ricci almost solitons, two new characterizations have been shown, and in the special case $\alpha=0$ (which includes Yamabe solitons, Yamabe quasisolitons and conformal gradient solitons) a very precise description of the metric in this situation has been provided. Working in a more general setting is often better. For instance, in~\cite{CMMR} it was already possible to recover a local version of the results in Cao-Chen~\cite{Cao and Chen 2013} and Cao et al.~\cite{Cao et al. 2014}.

Our purpose is to study some cases of the Einstein-type manifolds which were not addressed in~\cite{CMMR}. More precisely, we focus our analysis on the gradient case with $\beta\neq0$, which includes both degenerate and nondegenerate cases. In particular, the nondegeneracy condition $\beta^2\neq(n-2)\alpha\mu$ is crucial in the proof of Lemma~\ref{lemma3}. In this regard, if $f$ is constant in Eq.~\eqref{1-2} we say that an Einstein-type structure is trivial, once that $(M^n,g)$ becomes an Einstein manifold. However, the converse is generally false. Indeed, under the Einstein assumption we prove in Lemma~\ref{lemma1} that the positive function $u=e^{\frac{\mu}{\beta}f}$ provides a special concircular field on $(M^n,g)$. We refer to Theorem~2 in Tashiro~\cite{Tashiro} to determine the structure of complete Riemannian manifolds admitting such a field. There are other settings in which we may find nontrivial gradient Einstein-type structures on Einstein manifolds, see~\cite{CMMR}.

This note consists of two sections. In the first one, we study the compact case where we derive a general formula for any gradient Einstein-type manifold with both $\beta$ and $\mu$ nonzero, see Eq.~\eqref{EqMain}. As an application we prove Theorem~\ref{thmA} which is a rigidity result on the class of the nontrivial, compact, gradient Einstein-type manifolds of constant scalar curvature. Its proof is the inspiration for the construction of the examples as well as for the establishment of the desired results. In the second one, under some natural geometric assumptions, we prove Theorem~\ref{thmNoncompact} which is a rigidity result on the class of the nontrivial, noncompact, gradiente Einstein-type manifolds of constant scalar curvature. Theorems~\ref{thmA} and \ref{thmNoncompact} are motivated by the corresponding results for generalized $m$--quasi Einstein manifolds proven in Barros-Gomes~\cite{bjng} and Barros-Ribeiro~\cite{BR}, respectively. The next step is to study conditions for a noncompact gradient Einstein-type manifold to be an Einstein manifold. This is the subject of study of Theorem~\ref{thm-Homo}.

\section{Compact case}

The first theorem of this note is a rigidity result on the standard sphere. Its proof is motivated by the corresponding result for generalized $m$--quasi Einstein manifolds proven in Barros-Gomes~\cite{bjng}. This latter can be obtained making $(\alpha,\beta,\mu,\rho)=(1,1,-1/m,0)$ in Eq.~\eqref{1-2}. We emphasize that we only require $\beta\neq0$ and both $\alpha$ and $\mu$ can happen to be zeros. However, to prove our Theorem~\ref{thmA}, it is enough to consider $\mu\neq0$, as the $\mu=\alpha=0$ case has already been covered in the Obata's theorem~\cite{obata}. Indeed, the existence of a nontrivial gradient conformal vector field on a compact Riemannian manifold of constant scalar curvature implies that the manifold is isometric to a standard sphere. For more detail, the reader may look the Yano's book~\cite[p.~54]{yano}. In the case of $\mu=0$ and $\alpha\neq0$ it is enough to use Corollary~1 in Barros-Ribeiro~\cite{br2} which says that a nontrivial compact gradient Ricci almost soliton of constant scalar curvature is isometric to a standard sphere.
\begin{theorem}\label{thmA}
Let $(M^n,g)$ be a nontrivial, compact, gradient Einstein-type manifold of constant scalar curvature with both $\beta$ and $\mu$ nonzero. Then, $(M^n,g)$ is isometric to the standard sphere $\Bbb{S}^n(r)$. Moreover, up to rescaling and constant, the potential function is given by $f=\frac{\beta}{\mu}\ln\big(\tau-\frac{h_v}{n} \big),$ where $\tau$ is a real parameter lying in $(\frac{1}{n},+\infty)$ and $h_{v}$ is the height function on the unitary sphere $\Bbb{S}^{n}$.
\end{theorem}
\begin{proof}
Begin by considering the smooth function $u=e^{\frac{\mu}{\beta}f}$ for which are valid
\begin{equation*}
du=\frac{\mu}{\beta}udf \quad\mbox{and}\quad \nabla^2u = \frac{\mu}{\beta}u(\nabla^2f+\frac{\mu}{\beta}df\otimes df).
\end{equation*}
Thus, Eq.~\eqref{1-2} is equivalent to
\begin{equation}\label{1-3}
\frac{\alpha}{\beta}Ric + \frac{\beta}{\mu u}\nabla^2 u=\tilde\lambda g,
\end{equation}
where $\tilde\lambda=\frac{1}{\beta}(\rho S+\lambda)$.

Now, we can proceed as in~\cite{bjng} in order to prove the identity
\begin{equation}\label{EqMain}
{\rm div}\big(\mathring{Ric}(\nabla u)\big)=\frac{n-2}{2n}\mathscr{L}_{\nabla u}S-\frac{\alpha\mu}{\beta^2}u\|\mathring{Ric}\|^2,
\end{equation}
which is true for any gradient Einstein-type manifold with both $\beta$ and $\mu$ nonzero, where $\mathring{Ric}=Ric-\frac{S}{n}g$. Indeed, from the definition of the divergence we have
\begin{equation}\label{eqm1}
{\rm div}\big(\mathring{Ric}(\nabla u)\big)={\rm div}(\mathring{Ric})(\nabla u) + \langle \nabla^2 u,\mathring{Ric}\rangle.
\end{equation}
From the second contracted Bianchi identity we get
\begin{equation}\label{eqm2}
{\rm div}(\mathring{Ric})(\nabla u)= \frac{n-2}{2n}\langle \nabla S,\nabla u\rangle.
\end{equation}
Since $\langle g,\mathring{Ric}\rangle=0$, we use \eqref{1-3} to compute
\begin{equation}\label{eqm3}
\langle \nabla^2 u,\mathring{Ric}\rangle = \frac{\mu}{\beta}u\langle\tilde\lambda g-\frac{\alpha}{\beta}Ric,\mathring{Ric}\rangle=-\frac{\alpha\mu}{\beta^2}u\|\mathring{Ric}\|^2.
\end{equation}
Inserting \eqref{eqm2} and \eqref{eqm3} into \eqref{eqm1} we immediately deduce Eq.~\eqref{EqMain}.

We now assume that $S$ is constant and $M^n$ is compact. Moreover, without loss of generality, we also assume $\alpha\neq0$. So, it follows from Eq.~\eqref{EqMain} that $(M^n,g)$ is an Einstein manifold. This implies that $\nabla u$ is a nontrivial gradient conformal vector field, see Eq.~\eqref{1-3}. Applying the Obata's theorem~\cite{obata} we conclude that $M^n$ is isometric to the standard sphere $\Bbb{S}^{n}(r)$, and we can write
\begin{equation*}
\mathscr{L}_{\nabla u}g=2\psi g,
\end{equation*}
where the conformal factor $\psi=\frac{\Delta u}{n}$ satisfies the equation $\Delta \psi+\frac{S}{n-1}\psi=0$, see e.g. Ishihara-Tashiro~\cite{Ishihara-Tashiro}, Yano~\cite{yano2} or \cite[p.~28]{yano}. Rescaling the metric we can assume that $S=n(n-1)$. So, $\Delta u$ is the first eigenvalue of the unitary sphere $\Bbb{S}^{n}.$ Whence, there is a fixed vector $v\in \Bbb{S}^{n}$ such that $\Delta u =h_{v}=-\frac{1}{n} \Delta h_{v}$ (see~\cite{Berger}), which gives $u=\tau-\frac{1}{n} h_{v}$. When $\alpha=0$, we can use our latter argument, because it is immediate from Eq.~\eqref{1-3} that $\nabla u$ is a nontrivial gradient conformal vector field on a compact Riemannian manifold $(M^n,g)$ of constant scalar curvature. Note that this is sufficient to complete our proof.
\end{proof}

We identify Eq.~\eqref{EqMain} as the main formula for analyzing the noncompact case. Moreover, the equivalence between equations \eqref{1-2} and \eqref{1-3} as well as the obtainment of the function $u$ in the previous proof give us free of charge the next examples.

\begin{example}\label{ExM(c)}
Let $(\Bbb{M}^n(c),g_o)$ be the standard sphere $\Bbb{S}^n$ or the hyperbolic space $\Bbb{H}^n$ for $c=1$ or $c=-1$, respectively. We denote by $h_v$ the height function with respect to a fixed unit vector $v\in\Bbb{R}^{n+1}$. For each real numbers $\alpha,\rho$ and nonzero real numbers $\beta,\mu$, the functions $\lambda=c[-\rho n(n-1) +\alpha(n-1) + \frac{\beta^2}{\mu}\frac{(\tau-u)}{u}]$ and $f=\frac{\beta}{\mu} \ln(u)$ parameterize $(\Bbb{M}^n(c),g_{\circ})$ with a nontrivial gradient Einstein-type structure, where $u=\tau-\frac{c}{n}h_v$ and $\tau\in(\frac{c}{n},+\infty)$. Indeed, it is enough to recall that $Ric=c(n-1)g_o$ and $\nabla^2h_v=-ch_vg_o$, and thus conclude our statement by straightforward computation from Eq.~\eqref{1-3}.
\end{example}

\begin{example}\label{ExM(0)}
Let $(\Bbb{R}^n,g_o)$ be the Euclidean space. For each real numbers $\alpha,\rho$ and nonzero real numbers $\beta,\mu$, the functions $f=\frac{\beta}{\mu} \ln(u)$ and $\lambda=2\frac{\beta^2}{\mu u}$ parameterize $(\Bbb{R}^n,g_{\circ})$ with a nontrivial gradient Einstein-type structure, where $u=\|x\|^2+\tau$ and $\tau$ is a positive constant. Indeed, it is enough to recall that $\nabla^2\|x\|^2=2g_o$, and thus we again conclude our statement by straightforward computation from Eq.~\eqref{1-3}.
\end{example}

We point out that the previous examples and Lemma~\ref{lemma1} below are motivated by the corresponding results for generalized $m$--quasi Einstein manifolds proven in Barros-Ribeiro~\cite{BR}.

\section{Noncompact case}
Our purpose here is, under some geometric assumptions, to establish a rigidity result on a noncompact gradiente Einstein-type manifolds. The first step is to prove the following lemma.
\begin{lemma}\label{lemma1}
Let $(M^n,g)$ be a nontrivial gradient Einstein-type manifold with potential function $f$, and both $\beta$ and $\mu$ nonzero. If $(M^n,g)$ is an Einstein manifold, then the positive smooth function $u=e^{\frac{\mu}{\beta}f}$ provides a special concircular field on $(M^n,g)$. More precisely, the function $u$ satisfies the equation
\begin{equation}\label{EqConCirc}
\nabla^2 u = (-ku+b)g,
\end{equation}
with constant coefficients $k=\frac{S}{n(n-1)}$ and $b=\frac{\mu}{\beta}c$.
\end{lemma}
\begin{proof}
If $(M^n,g)$ is an Einstein manifold, i.e., $Ric=\frac{S}{n}g$, then from \eqref{1-3} we have
\begin{equation}\label{1-4}
\nabla^2 u = \frac{\mu}{\beta}\left(\tilde\lambda u -\frac{\alpha}{\beta}\frac{S}{n}u\right) g \quad\mbox{and}\quad \Delta u = \frac{\mu}{\beta}\left(\tilde\lambda u -\frac{\alpha}{\beta}\frac{S}{n}u\right)n.
\end{equation}

Two well-known general facts in the literature are
\begin{equation*}
Ric(\nabla \psi)+\nabla\Delta \psi = {\rm div}\nabla^2\psi \quad\mbox{and} \quad {\rm div}(\varphi g)=\nabla\varphi
\end{equation*}
for all $\psi,\varphi\in C^\infty(M)$.

Applying these facts for $\psi=u$ and $\varphi= \frac{\mu}{\beta}\left(\tilde\lambda u -\frac{\alpha}{\beta}\frac{S}{n}u\right)$ together with the identities of \eqref{1-4} we obtain
\begin{equation*}
\frac{S}{n}\nabla u+\frac{\mu}{\beta}\nabla\left(\tilde\lambda n u -\frac{\alpha}{\beta}Su\right) = \frac{\mu}{\beta}\nabla \left(\tilde\lambda u -\frac{\alpha}{\beta}\frac{S}{n}u\right).
\end{equation*}
Hence,
\begin{equation*}
\nabla\left(\frac{\beta^2S-\mu\alpha(n-1)S}{n(n-1)\beta^2}u+\frac{\mu}{\beta}\tilde\lambda u\right)= 0.
\end{equation*}
By connectedness of $M^n$ there is a constant $c$ such that
\begin{equation}\label{1-5}
\frac{\mu}{\beta}\tilde\lambda u = -\frac{\beta^2S-\mu\alpha(n-1)S}{n(n-1)\beta^2}u + c.
\end{equation}
Replacing \eqref{1-5} in the first identity of \eqref{1-4}, making $k=\frac{S}{n(n-1)}$ and $b=\frac{\mu}{\beta}c$, it is obtained the required Eq.~\eqref{EqConCirc}.
\end{proof}

The second step is to use the following well-known general extension of the Stoke's theorem in the noncompact case, namely:

\begin{lemma}[Karp's theorem~\cite{Karp}]
Let $(M^n,g)$ be a complete noncompact Riemannian manifold. Consider the geodesic ball $B(r)$ of radius $r$ centered at some fixed $p\in M^n$ and a vector field $X$ such that
\begin{equation*}
\liminf_{r\to\infty}\frac{1}{r}\int_{B(2r)\backslash B(r)}\|X\| dvol_g=0.
\end{equation*}
If ${\rm div} X$ has an integral (i.e., if either (${\rm div} X)^+$ or $({\rm div} X)^-$ is integrable), then
$\int_M {\rm div} X dvol_g = 0.$ In particular, if ${\rm div} X$ does not change sign outside some compact set, then $\int_M {\rm div} X dvol_g = 0.$
\end{lemma}

Since $\beta\neq0$, for the next two theorems we also consider both $\mu$ and $\alpha$ nonzero, as the $\mu=\alpha=0$ case has already been covered in Tashiro~\cite{Tashiro}. Moreover, Theorem~1.4 and Corollary~1.5 in~\cite{CMMR} are local results that provide a very precise description of a gradient Einstein-type manifold with $\alpha=0$. In the case of $\mu=0$ and $\alpha\neq0$ we refer to Calvi\~no-Louzao et al.~\cite{EMER} from which we know that a locally homogeneous proper Ricci almost soliton is either of constant sectional curvature or a product $\Bbb{R}\times N(c)$, where $N(c)$ is a space of constant curvature $c$.

\begin{theorem}\label{thmNoncompact}
Let $(M^n,g)$ be a nontrivial, noncompact, gradient Einstein-type manifold of constant scalar curvature with $\beta$, $\mu$ and $\alpha$ nonzero. Consider the geodesic ball $B(r)$ of radius $r$ centered at some fixed $p\in M^n$. In addition, suppose that at least one of the following conditions holds:
\begin{enumerate}
\item \label{thm2-1}$\displaystyle\liminf_{r\to\infty}\frac{1}{r}\int_{B(2r)\backslash B(r)}\|\mathring{Ric}(\nabla u)\|dvol_g=0.$
\item \label{thm2-2}$\displaystyle\liminf_{r\to\infty}\frac{1}{r}\int_{B(2r)\backslash B(r)}\|\nabla u\|dvol_g=0$ and Ricci curvature is upper bounded.
\item \label{thm2-3}There is $C>0$ such that $vol(B(r))\leq Cr^q$ for $r\geq1$ and $\mathring{Ric}(\nabla u)$ lies in $L^p(M^n, dvol_g)$ where $1/p+1/q=1$ and $q>1$.
\item \label{thm2-4}There is $C>0$ such that $vol(B(r))\leq Cr$ for $r\geq1$ and $\|X\|\to0$ uniformly at infinity in $M^n$.
\end{enumerate}
Then, $(M^n,g)$ is an Einstein manifold of nonpositive scalar curvature $S$ and $u$ has at most one stationary point. More precisely:
\begin{enumerate}
\item[(i)] If $S=0$, then $\lambda$ has no zeros and $(M^n,g)$ is isometric to a Euclidean space.
\item[(ii)] If $S<0$ and $u$ has only one stationary point, then $(M^n,g)$ is isometric to a hyperbolic space.
\item[(iii)] If $S<0$ and $u$ has no stationary point, then $(M^n,g)$ is isometric to a warped product $\Bbb{R}\times_{\varphi}\Bbb{F}$, where $\Bbb{F}$ is a complete Einstein manifold, and $\varphi$ is a solution of the differential equation $\ddot{\varphi}+\frac{S}{n(n-1)}\varphi=0$ on the real line $\Bbb{R}$.
\end{enumerate}
\end{theorem}
\begin{proof}
Instead of using the constancy assumption on $S$, we shall prove the theorem under the weaker condition that $\langle\nabla u,\nabla S\rangle\leq0$ on $M^n$, if $\alpha\mu>0$, or $\langle\nabla u,\nabla S\rangle\geq0$ on $M^n$, if $\alpha\mu<0$. Anyway, ${\rm div}(\mathring{Ric}(\nabla u))$ given by Eq.~\eqref{EqMain} does not change sign on $M^n$.

\vspace{0.1cm}
\noindent \textbf{Item~\eqref{thm2-1}:} By integrating of Eq.~\eqref{EqMain} on $M^n$, we deduce immediately from Karp's theorem that $M^n$ is an Einstein manifold. Thus, $S$ must be constant, and as $M^n$ is a complete noncompact, we must have $S\leq0$ by Bonnet-Myers theorem.

\vspace{0.1cm}
\noindent \textbf{Item~\eqref{thm2-2}:} In this case, note that there is a constant $C>0$ such that $Ric\leq C$ and
\begin{equation*}
\|\mathring{Ric}(\nabla u)\|^{2}\leq\|\mathring{Ric}\|^{2}\|\nabla u\|^{2}=\left(\|Ric\|^{2}-\dfrac{S^{2}}{n}\right)\|\nabla u\|^{2}\leq C\|\nabla u\|^{2}.
\end{equation*}
Hence, Item~\eqref{thm2-2} implies Item~\eqref{thm2-1}.

\vspace{0.1cm}
\noindent \textbf{Items~\eqref{thm2-3} and \eqref{thm2-4}:} Both are also a consequence of Item~\eqref{thm2-1}, see~\cite[Corollary~1]{Karp}.

Now, as a consequence of Lemma~\ref{lemma1}, in all cases the function $u$ satisfies Eq.~\eqref{EqConCirc} from which we apply Theorem~2 in~\cite{Tashiro} to deduce that $M^n$ is as in {\rm (i)-(iii)}. More precisely, we have
\begin{equation}\label{EqConCirc-thm2}
\nabla^2u=\left(-\frac{S}{n(n-1)}u+\frac{\mu}{\beta}c\right)g,
\end{equation}
were $c$ is a constant given by Eq.~\eqref{1-5}.

\vspace{0.1cm}
\noindent\textbf{(a)} $S=0$ case: Here, the proof depends on the nullity and not of the function $\lambda$. From \eqref{1-5} we get $c=\frac{\mu}{\beta^2}\lambda u$. Thus, if $c=0$, then $\lambda\equiv0$. Conversely, if there is $p\in M$ such that $\lambda(p)=0$, then $c=0$ and $\lambda\equiv0$. And, if $\lambda$ has no zeros, we must have $c\neq0$.

\vspace{0.1cm}
\noindent \textbf{Item~{\rm(i)}:} Suppose that there is $p\in M^n$ such that $\lambda(p)=0$. Then, we have $\lambda\equiv0$ and $\nabla^2u=0$. But, since $M^n$ is Ricci flat it is not possible to find a positive nonconstant harmonic function $u$ on $M^n$, see Yau~\cite{Yau}. Consequently, $\lambda$ has no zeros. So, we have $c\neq0$ and $\nabla^2u=\frac{\mu}{\beta}cg$. Hence, $(M^n,g)$ is isometric to a Euclidean space. Since $u$ must be positive, we can take, e.g., $u(x)=\|x\|^2+\tau$, with $\tau>0$. In this case, the Einstein-type structure on $M^n$ is given by Example {\rm\ref{ExM(0)}}.

\vspace{0.1cm}
\noindent\textbf{(b)} $S<0$ case: Here, the proof depends on the number of stationary points of the concircular scalar field $u$ which must be at most one, see~\cite[Theorem~1]{Tashiro}.

\vspace{0.1cm}
\noindent\textbf{Item~{\rm(ii)}:} If $u$ has only one stationary point, then $(M^n,g)$ is isometric to a hyperbolic space. Now let us suppose that $(M^n,g)$ is isometric to $(\Bbb{H}^n(-1),g_o)$, and we can choose any point $p$ in $\Bbb{H}^n(-1)$. It is well-known that the function $\tilde{u}(q):=\cosh(r)$ satisfies $\nabla^2\tilde{u}=\coth(r)g_o$ on $\Bbb{H}^n(-1)$, where $r=dist(p,q)$. Then the conformal factor gives rise to a nontrivial solution of Eq.~\eqref{EqConCirc-thm2} (see e.g.~\cite[p.~28]{yano}), namely $u(q)=\coth(r)$, being $p$ the stationary point. Another positive nontrivial solution of Eq.~\eqref{EqConCirc-thm2} on $\Bbb{H}^n(-1)$ has been presented in Example~{\rm\ref{ExM(c)}}.

\vspace{0.1cm}
\noindent\textbf{Item~{\rm(iii)}:} If $u$ has no stationary point, then $(M^n,g)$ is isometric to a warped product $\Bbb{R}\times_{\varphi}\Bbb{F}$, where $\Bbb{F}$ is a complete Einstein manifold and $\varphi$ is a positive solution of $\ddot{\varphi}+k\varphi=0$ on $\Bbb{R}$ with $k=\frac{S}{n(n-1)}$ (see also Masahiko~\cite[Theorem G]{masahiko}). The value of $k$ and the fact of $\Bbb{F}$ to be Einstein, it follows from Bishop-O'Neill's formulas~\cite{BO} or O'Neill~\cite[Corollary~43]{oneill}. Finally, we mention that a nontrivial Einstein-type structure on $\Bbb{R}\times_{\varphi}\Bbb{F}$ also occur, see Example~\ref{ExWP}. This completes the proof of theorem.
\end{proof}

\begin{example}\label{ExWP}
We consider a positive solution $\varphi: \Bbb{R}\to\Bbb{R}^+$ of the differential equation $\ddot{\xi}+k\xi=0$, given by
\begin{equation*}
\varphi(t)=\frac{a}{\sqrt{-k}}\sinh(\sqrt{-k}\,t)+\sqrt{\frac {a^2+l}{-k}}\cosh(\sqrt{-k}\,t),
\end{equation*}
where $a\neq0$, $l>0$ and $k<0$ are constants. Let $(\Bbb{F},g_{\Bbb{F}})$ be an $(n-1)$--dimensional complete Einstein manifold with $Ric_{g_{\Bbb{F}}}= -(n-2)lg_{\Bbb{F}}$, $n\geq3$. From Bishop-O'Neill's formulas or, alternatively, from Lemma~2.1 in~\cite{prrs}, we can construct an Einstein warped product $M^n=\Bbb{R}\times_\varphi\Bbb{F}$ endowed with the metric $g=dt^2 + \varphi(t)^2g_{\Bbb{F}}$ and Ricci tensor $Ric_g=(n-1)kg$, where $dt^2$ stands for the canonical metric of $\Bbb{R}$. Moreover, $u(t,p)=\varphi(t)$ is a positive function without critical points satisfying the differential equation $\nabla^2u+kug=0$ on $M^n$. By using Eq.~\eqref{1-3} we obtain
\begin{equation*}
\lambda=-\rho n(n-1)k + [\mu\alpha(n-1)-\beta^2]k/\mu.
\end{equation*}
Thus, the function $f=\frac{\beta}{\mu} \ln(u)$ and the constant $\lambda$ parameterize $M^n$ with a nontrivial gradient Einstein-type structure, for each real numbers $\rho,\alpha$ and nonzero real numbers $\beta,\mu$.
\end{example}

Now, we will study some conditions for a noncompact gradient Einstein-type manifold (with $\beta\neq0$) to be an Einstein manifold. Obviously, some restrictions on the parameters $\mu$ and $\alpha$ must be established. The one that seems more natural is to consider both $\mu$ and $\alpha$ nonzero. In this case, Eq.~\eqref{1-2} is equivalent to
\begin{equation}\label{Eq-hARS}
Ric + h\nabla^2 u=\ell g,
\end{equation}
where $u=e^{\frac{\mu}{\beta}f}$, $h=\frac{\beta^2}{\mu\alpha}\frac{1}{u}$ and $\ell=\frac{1}{\alpha}(\rho S+\lambda)$. Now, we can apply the approach of the generalized $m$--quasi Einstein metrics.

In what follows we assume that $h,\ell$ and $u$ are arbitrary smooth functions satisfying Eq.~\eqref{Eq-hARS} on a Riemannian manifold $(M^n,g)$. Then,
\begin{eqnarray*}
\frac{1}{2} dS &=&\dv Ric \,=\, d\ell - h \dv\nabla^2u -\nabla^2u(\nabla h,\cdot)\\
&=&d\ell -hRic(\nabla u,\cdot)-hd\Delta u -\nabla^2u(\nabla h,\cdot).
\end{eqnarray*}
Since $d(h\Delta u)=hd\Delta u+\Delta u dh$ and $h\Delta u=n\ell-S$, we immediately obtain
\begin{equation}\label{Eq-hARS-2}
(n-1)d\ell=\frac{1}{2}dS - hRic(\nabla u,\cdot) + \Delta u dh-\nabla^2u(\nabla h,\cdot).
\end{equation}
On the other hand, from Eq.~\eqref{Eq-hARS} we get
\begin{equation}\label{Eq-hARS-3}
hRic(\nabla u,\cdot) = \ell hdu-h^2\nabla^2u(\nabla u,\cdot) =\ell hdu-\frac{h^2}{2}d|\nabla u|^2.
\end{equation}
Replacing \eqref{Eq-hARS-3} into \eqref{Eq-hARS-2} we obtain
\begin{equation}\label{Eq-hARS-4}
(n-1)d\ell=\frac{1}{2}dS-\ell hdu +\frac{h^2}{2}d|\nabla u|^2 + \Delta udh - \nabla^2u(\nabla h,\cdot).
\end{equation}
In particular, for $h=\frac{c}{u}$, where $c=\frac{\beta^2}{\mu\alpha}$, we have
\begin{equation}\label{Eq-dh}
dh=-\frac{c}{u^2}du\quad \mbox{and}\quad c\Delta u=(n\ell-S)u.
\end{equation}
Substituting \eqref{Eq-dh} into \eqref{Eq-hARS-4} yields
\begin{eqnarray*}
(n-1)d\ell &=&\frac{1}{2}dS-\frac{c}{u}\ell du +\frac{c^2}{2u^2}d|\nabla u|^2 -\frac{c\Delta u}{u^2}du + \frac{c}{u^2}\nabla^2u(\nabla u,\cdot)\\
&=&\frac{1}{2}dS-\frac{c}{u}\ell du -\frac{n\ell-S}{u}du  +\frac{c^2}{2u^2}d|\nabla u|^2 + \frac{c}{2u^2}d|\nabla u|^2.
\end{eqnarray*}
Therefore
\begin{equation*}
(n-1)u^2d\ell =\frac{u^2}{2}dS-cu\ell du -(n\ell-S)udu + \left(\frac{c^2+c}{2}\right)d|\nabla u|^2.
\end{equation*}
Applying the exterior derivative, we deduce
\begin{equation*}
2(n-1)udu\wedge d\ell =udu\wedge dS-cud\ell\wedge du -nud\ell\wedge du+udS\wedge du.
\end{equation*}
Since $u>0$ and $c=\frac{\beta^2}{\mu\alpha}$ we conclude that
\begin{equation}\label{EqWedge}
[\beta^2-(n-2)\mu\alpha]du\wedge d\ell =0.
\end{equation}
Remarkably, it is enough to consider the nondegeneracy condition $\beta^2-(n-2)\mu\alpha\neq0$ to derive a good relation between $u$ and $\ell$ as follows.

\begin{lemma}\label{lemma3}
Let $(M^n,g)$ be a nontrivial nondegenerate gradient Einstein-type manifold with both $\mu$ and $\alpha$ nonzero. Then, $\nabla\ell=\psi\nabla u$ for some $\psi\in C^\infty(M)$, where $\ell=\frac{1}{\alpha}(\rho S+\lambda)$ and $u=e^{\frac{\mu}{\beta}f}$.
\end{lemma}
\begin{proof}
Recall that the nondegeneracy condition means $\beta^2-(n-2)\mu\alpha\neq0$ and $\beta\neq0$. As both $\mu$ and $\alpha$ are nonzero we can conclude the statement of the lemma immediately from Eq.~\eqref{EqWedge}.
\end{proof}

We point out that Lemma~\ref{lemma3} is an extension of the $m$--quasi-Einstein manifold case proved very recently in Hu-Li-Zhai~\cite{HLX2} to the nondegenerate gradient Einstein-type manifold case. For the Ricci almost soliton case a result corresponding to the referred lemma has been obtained in~\cite{prrs}. In~\cite{HLX2}, a technique analogous to that in~\cite{prrs} was used.

We now proceed by using the equations \eqref{Eq-hARS}, \eqref{Eq-hARS-2} and \eqref{Eq-dh} to compute
\begin{eqnarray*}
\frac{c}{u}Ric(\nabla u)&=&\frac{1}{2}\nabla S-(n-1)\nabla\ell - \frac{1}{u}\frac{c}{u}\Delta u\nabla u+\frac{1}{u}\frac{c}{u}\nabla^2u(\nabla u)\\
&=&\frac{1}{2}\nabla S-(n-1)\nabla\ell - (n\ell-S)\frac{1}{u}\nabla u+\frac{\ell}{u}\nabla u-\frac{1}{u}Ric(\nabla u).
\end{eqnarray*}
By regrouping
\begin{equation}\label{EigRic-Aux}
\frac{\beta^2+\mu\alpha}{\mu\alpha}Ric(\nabla u)=\frac{u}{2}\nabla S-(n-1)u\nabla\ell - [(n-1)\ell-S]\nabla u.
\end{equation}

As an application of Eq.~\eqref{EigRic-Aux} and Lemma~\ref{lemma3} we will deduce our next theorem. The most appropriate setting to do this is to consider the class of the proper Einstein-type manifolds, for which $\lambda$ is nonconstant on $M^n$. Since $n\geq3$, it is easy to see that a proper Einstein-type manifold must be nontrivial. Examples~\ref{ExM(c)} and \ref{ExM(0)} are Einstein manifolds supporting a proper gradient Einstein-type structure, while Example~\ref{ExWP} is an Einstein manifold supporting a nontrivial, nonproper, gradient Einstein-type structure.

\begin{theorem}\label{thm-Homo}
Let $(M^n,g)$ be a homogeneous, proper, noncompact, nondegenerate, gradient Einstein-type manifold with both $\mu$ and $\alpha$ nonzero. If $\beta^2\neq-\mu\alpha$, then $(M^n,g)$ is an Einstein manifold.
\end{theorem}
\begin{proof}
Instead of using the homogeneous assumption, we shall prove the theorem under the weaker condition that the Ricci curvatures of $(M^n,g)$ are constant, from which we shall also have the constancy of the scalar curvature. By Lemma~\ref{lemma3}, $\nabla\ell=\psi\nabla u$ for some $\psi\in C^\infty(M)$. So, using Eq.~\eqref{EigRic-Aux} we have
\begin{equation}\label{EigRic}
\frac{\beta^2+\mu\alpha}{\mu\alpha}Ric(\nabla u)=-\left[(n-1)u\psi + (n-1)\ell-S\right]\nabla u.
\end{equation}
Since $\beta^2\neq-\mu\alpha$, it follows that $\nabla u$ is an eigenvector of the Ricci tensor. Hence, Eq.~\eqref{EigRic} becomes
\begin{equation*}
Ric(\nabla u)=k\nabla u,
\end{equation*}
where $k=- \frac{\mu\alpha}{\beta^2+\mu\alpha}\left[(n-1)u\psi + (n-1)\ell-S\right]$ must be a constant. It is also true that
\begin{equation*}
\mathring{Ric}(\nabla u)=\bar{k}\nabla u, \quad \bar{k}=k-\frac{S}{n}.
\end{equation*}
Combining this latter equation and \eqref{EqMain} we get
\begin{equation*}
\bar{k}\Delta u={\rm div}\big(\mathring{Ric}(\nabla u)\big)=-\frac{\alpha\mu}{\beta^2}u\|\mathring{Ric}\|^2.
\end{equation*}
Hence, by using of the second equation in \eqref{Eq-dh} we obtain
\begin{equation}\label{EigRic1}
\bar{k}(n\ell-S)=-\|\mathring{Ric}\|^2.
\end{equation}
By hypothesis, $\|\mathring{Ric}\|^2$ is constant and $\lambda$ is nonconstant. So, we can conclude, by simple analysis on Eq.~\eqref{EigRic1}, that $(M^n,g)$ is an Einstein manifold.
\end{proof}

Theorem~\ref{thm-Homo} has been recently proven by Hu-Li-Zhai~\cite{HLX2} for the $m$--quasi-Einstein manifold case, while the Ricci almost soliton case had been proven shortly before in~\cite{EMER}. Technically speaking, allowing $m$ to be a real number in~\cite{HLX2}, then the $\beta^2\neq-\mu\alpha$ case can be obtained from Theorem~1.2 in~\cite{HLX2}. However, by virtue of the generality of our Theorem~\ref{thmA}, it should be noted that is not necessary to consider compactness under the hypothesis of the constancy of Ricci curvatures. Furthermore, we provide some simplifications in comparison with the current literature.

\section{Acknowledgements}
The author would like to express his sincere thanks to Dragomir M. Tsonev (UFAM, Brazil) and Manoel V.M. Neto (UFPI, Brazil) for useful comments, discussions and constant encouragement as well as to Department of Mathematics at Lehigh University, where this note was written. The author is grateful to Huai-Dong Cao and Mary Ann for the warm hospitality and their constant encouragement. He also thanks the two anonymous referees for careful reading this note.

\end{document}